\documentclass[11pt]{article}

\usepackage[utf8]{inputenc}

\usepackage{geometry}
\usepackage{graphicx}

\usepackage{booktabs} 
\usepackage{array} 
\usepackage{paralist} 
\usepackage{verbatim} 
\usepackage{subfig} 

\usepackage{fancyhdr} 
\usepackage[nottoc,notlof,notlot]{tocbibind}
\usepackage[titles,subfigure]{tocloft}

\usepackage{amsmath,amsfonts,amsthm,mathrsfs,amssymb,cite}
\usepackage[usenames]{color}

\newcommand{\R}{{\mathbb R}}

\newcommand{\N}{{\mathbb N}}
\newcommand{\C}{{\mathbb C}}

\newcommand{\be}{\begin{eqnarray}}
\newcommand{\ben}{\begin{eqnarray*}}
\newcommand{\en}{\end{eqnarray}}
\newcommand{\enn}{\end{eqnarray*}}

\newtheorem{thm}{Theorem}[section]

\newtheorem{lem}{Lemma}[section]

\theoremstyle{definition}

\theoremstyle{remark}

\newtheorem{rem}{Remark}[section]
\numberwithin{equation}{section}

\title{\bf Uniqueness in Determining Refractive Indices by Formally-determined Far-field Data}
\author {Guanghui Hu\thanks{Weierstrass Institute (WIAS), Mohrenstr. 39,
10117 Berlin,
Germany. Email: {\tt hu@wias-berlin.de}},\ \ \ Hongyu Liu\thanks{Department of Mathematics and Statistics, University of North Carolina, Charlotte, NC 28223, USA. }}

\date{} 

\begin{document}
\maketitle

\begin{abstract}

We present two uniqueness results for the inverse problem of determining an index of refraction by the corresponding acoustic far-field measurement encoded into the scattering amplitude. The first one is a local uniqueness in determining a variable index of refraction by the fixed incident-direction scattering amplitude. The inverse problem is formally posed with such measurement data. The second one is a global uniqueness in determining a constant refractive index by a single far-field measurement. The arguments are based on the study of certain nonlinear and non-selfadjoint interior transmission eigenvalue problems.
\end{abstract}

\section{Introduction}\label{sec:intro}

This note is concerned with the inverse acoustic scattering problem of recovering the {\it refractive index} of an inhomogeneous medium. Suppose in the homogeneous space, there is an inhomogeneity. In order to determine the inhomogeneity, one sends a certain detecting wave field. The propagation of the detecting wave field will be perturbed when meeting the inhomogeneity. The {\it perturbation} is the so-called {\it scattering} in the literature. The inverse problem is to recover the inhomogeneity by measuring the corresponding scattering wave field. The study of inverse scattering problems lies in the core of many areas of science and technology, such as radar and sonar, geophysical exploration, medical imaging, nondestructive testing, and remote sensing etc.; see \cite{AK1,AK2,CK,Isa,PiS02,Uhl} and the references therein.

Throughout, we shall take the incident field to be the time-harmonic acoustic plane wave of the form $u^i(x):=e^{ik x\cdot d}$, $x\in\mathbb{R}^N$ with $N\geq 2$, where $k\in\mathbb{R}_+$ is the wave number and $d\in\mathbb{S}^{N-1}$ is the impinging direction. The optical property of the homogeneous space is described an index of refraction $n(x)$, which is normalized to be $1$, whereas in an open domain $\Omega$ accommodating the inhomogeneity, it is assumed that $n(x)-1\equiv\hspace*{-3.5mm}\backslash\ 0$, $x\in\Omega$. We assume that $\Omega$ is a bounded Lipschitz domain with $\mathbb{R}^N\backslash\overline{\Omega}$ connected, and that the refractive index $n(x)\in L^\infty(\Omega)$ satisfying $\|n-1\|_{L^\infty(\Omega)}\geq \epsilon_0>0$. The wave propagation is governed by the following Helmholtz equation
\begin{equation}\label{eq:Helm}
\Delta u(x)+k^2 n^2(x) u(x)=0,\quad x\in\mathbb{R}^N,
\end{equation}
where $u(x)$ denotes the wave pressure. We seek a solution $u\in H_{loc}^1(\mathbb{R}^N)$ such that
\begin{equation}\label{eq:total}
u(x)=u^i(x)+u^s(x),\quad x\in\mathbb{R}^N\backslash\overline{\Omega}.
\end{equation}
$u^s(x)$ is called the {\it scattered wave field}, which satisfies the so-called {\it Sommerfeld radiation condition},
\begin{equation}\label{eq:radiation}
\lim_{|x|\rightarrow \infty} |x|^{\frac{N-1}{2}}\left\{ \frac{\partial u^s(x)}{\partial |x|}-ik u^s(x) \right\}=0.
\end{equation}
The well-posedness of the scattering system \eqref{eq:Helm}--\eqref{eq:radiation} is well understood (cf. \cite{Isa,LSSZ}). Particularly, $u^s$ admits the following asymptotic expansion
\begin{equation}\label{eq:farfield}
u^s(x)=\frac{e^{ik|x|}}{|x|^{\frac{N-1}{2}}} u_\infty(\hat x; d, k)+\mathcal{O}\left(\frac{1}{|x|^{\frac N 2}}\right),\quad |x|\rightarrow+\infty,
\end{equation}
which holds uniformly in all directions $\hat x:=x/|x|$, $x\in\mathbb{R}^N$. $u_\infty(\hat x; d, k)$ is known as the {\it scattering amplitude}. By the celebrated Rellich's Lemma (cf. \cite{CK}), the scattering amplitude $u_\infty(\hat x)$ encodes all the information of the scattered wave field $u^s(x)$. The inverse scattering problem that we consider in the present study is to recover the scatterer $(\Omega, n)$ by knowledge of $u_\infty(\hat x; d, k)$. If one introduces an operator $\mathcal{F}$ which sends the {\it scatterer} $(\Omega, n)$ to the corresponding scattering amplitude, then the inverse scattering problem can be abstractly formulated as the following operator equation
\begin{equation}\label{eq:oe}
\mathcal{F}(n)=u_\infty(\hat x; d, k).
\end{equation}
It is easy to verify that the operator equation \eqref{eq:oe} is nonlinear and also widely known to be ill-posed (cf. \cite{CK,Isa}).

One of the foundational issues in the inverse scattering problem is the uniqueness/identifiability: can one really identify the scatterer by the measurement? how many measurement data one should use for the identification? Mathematically, the uniqueness issue can be stated as follows. Let $n_1$ and $n_2$ be two refractive indices with the scattering amplitudes $u_\infty^1(\hat x; d, k)$ and $u_\infty^2(\hat x; d, k)$. Then
\begin{equation}\label{eq:uniqueness}
u_\infty^1(\hat x; d, k)=u_\infty^2(\hat x; d, k)\quad\mbox{if{f}}\quad n_1=n_2.
\end{equation}
The study on uniqueness is usually very difficult and challenging. Intensive efforts have been devoted to the unique determination of many inverse scattering problems and quite different technical treatments and mathematical theories are needed with different problems. Some significant developments have been achieved for several representative inverse scattering problems; see \cite{AK1,CK,Isa,LZ,U,Uhl} and the references therein. The classical uniqueness for the inverse problem \eqref{eq:oe} is due to Sylvester and Uhlmann (cf. \cite{SyU}; see also \cite{Nac,CK,Isa}). It is shown that in $\mathbb{R}^N$, $N\geq 3$, $u_\infty(\hat x; d, k)$ with a fixed $k\in\mathbb{R}_+$ and all $(\hat x, d)\in \mathbb{S}^{N-1}\times\mathbb{S}^{N-1}$ uniquely determines $n(x)$. Here, we note that the measurement data used are over-determined. Indeed, one can easily count the {\it dimensions} of the known (namely $u_\infty(\hat x; d, k)$) and the unknown (namely $n(x)$), are respectively $2N-2$ and $N$. By {\it dimension}, we mean the number of independent variables in a set. Clearly, one has $2N-2>N$ if $N\geq 3$. This motivates us to consider the unique determination of the refractive index function $n(x)$ by $u_\infty(\hat x; d, k)$ with a fixed $d\in\mathbb{S}^{N-1}$ and all $(\hat x, k)\in\mathbb{S}^{N-1}\times\mathbb{R}_+$. It is easily seen that the inverse scattering problem is formally posed with such measurement data. Since $u_\infty(\hat x; d, k)$ is (real) analytic with all its arguments and due to analytic continuation, the measurement data set could be replaced by $u_\infty(\hat x; d, k)$ for a fixed $d\in\mathbb{S}^{N-1}$, and all $(\hat x, k)\in\Lambda$, where $\Lambda$ is any open subset of $\mathbb{S}^{N-1}\times \mathbb{R}_+$. To our best knowledge, there is no uniqueness result in the literature by using such a formally-determined data set. Our study connects to that of the nonlinear and non-selfadjoint {\it interior transmission eigenvalue problems}, which was first introduced by Colton and Monk in \cite{CM}, and were recently extensively studied; see \cite{CC,CKP,PS,S} and the references therein. We conjecture that the uniqueness can be established in a very generic setting by using the formally-determined data set. However, in Section~\ref{sect:2}, we shall only present a local uniqueness result in determining a variable refractive index, by directly implementing the discreteness of interior transmission eigenvalues due to Sylvester in \cite{S}. Our main contribution in Section~\ref{sect:2} is to bridge the study on uniqueness of the inverse scattering problem and that on the interior transmission eigenvalue problems, which in our hope might open up some new directions of research in the relevant field. In Section~\ref{sect:3}, a much more interesting uniqueness result is derived in determining a constant refractive index by a single far-field measurement, namely $u_\infty(\hat x; d, k)$ with both $d$ and $k$ fixed, and all $\hat x\in\mathbb{S}^{N-1}$. This is based on showing that there is a lower bound for the positive interior transmission eigenvalues. The uniqueness by a single far-field measurement is extremely challenging for inverse scattering problems in the literature; see Section~\ref{sect:remark} for more discussion on this aspect.

\section{Local uniqueness in determining a variable refractive index}\label{sect:2}

Throughout, we let $n_*$ and $n^*$ be two positive constants such that $n_*<n^*$. Let $(\Omega, n_j)$, $j=1,2$, be two inhomogeneous media as described in Section~\ref{sec:intro} satisfying $n_*\leq n_j(x)\leq n^*$ for $x\in\Omega$. The scattering problem \eqref{eq:Helm}--\eqref{eq:radiation} corresponding to $(\Omega, n_j)$ can be easily formulated as the following transmission eigenvalue problem: find a pair of solutions $(u_j, u_j^s)\in H^1(\Omega)\times H_{loc}^1(\mathbb{R}^N\backslash\overline{\Omega})$ such that
\begin{equation}\label{eq:trans1}
\begin{cases}
\Delta u_j+k^2 n_j^2 u_j=0\hspace*{1cm} &\mbox{in\ \ $\Omega$},\\
\Delta u_j^s+k^2 u_j^s=0\hspace*{1cm} &\mbox{in\ \ $\mathbb{R}^N\backslash\overline{\Omega}$},\\
u_j=u_j^s+u^i,\ \ \frac{\partial u_j}{\partial \nu}=\frac{\partial u_j^s}{\partial \nu}+\frac{\partial u^i}{\partial\nu} &\mbox{on\ \ $\partial\Omega$},\\
\lim_{|x|\rightarrow \infty} |x|^{\frac{N-1}{2}}\left\{ \frac{\partial u_j^s(x)}{\partial |x|}-ik u_j^s(x) \right\}=0,
\end{cases}
\end{equation}
where $\nu$ denotes the exterior unit normal vector to $\partial\Omega$. We let $u_\infty^j(\hat x; d, k)$ be the scattering amplitude of \eqref{eq:trans1}. Suppose that for a fixed $d\in\mathbb{S}^{N-1}$ and a fixed $k\in\mathbb{R}_+$,
\begin{equation}\label{eq:equi}
u_\infty^1(\hat x; d, k)=u_\infty^2(\hat x; d, k)\quad\mbox{for\ \ $\hat x\in\mathbb{S}^{N-1}$}.
\end{equation}
Then by Rellich's Lemma (cf. \cite{CK}), we have
\begin{equation}\label{eq:equi2}
u^s_1(x)=u^s_2(x),\quad x\in\mathbb{R}^N\backslash\overline{\Omega},
\end{equation}
By \eqref{eq:trans1} and \eqref{eq:equi2}, it is easy to see that $(u_1, u_2)\in H^1(\Omega)\times H^1(\Omega)$ satisfies
\begin{equation}\label{eq:itep}
\begin{cases}
\Delta u_1+k^2 n_1^2 u_1=0\qquad & \mbox{in\ \ $\Omega$},\\
\Delta u_2+k^2 n_2^2 u_2=0\qquad & \mbox{in\ \ $\Omega$},\\
u_1=u_2,\ \ \frac{\partial u_1}{\partial\nu}=\frac{\partial u_2}{\partial\nu}\ & \mbox{on\ \ $\partial\Omega$}.
\end{cases}
\end{equation}
Next, we show that $(u_1, u_2)$ must be non-trivial solutions. It is easily seen that if either one of $u_1$ and $u_2$ is a trivial solution, then both of them are trivial solutions to \eqref{eq:itep}. Without loss of generality, we suppose that $u_1=0$. Then by \eqref{eq:trans1}, we have that
\[
 u^i+u_1^s=\partial (u^i+u_1^s)/\partial\nu=0 \quad\mbox{on}\quad \partial\Omega. 
 \]
 Hence, by Holmgren's Theorem (cf. \cite{CK}), one must have that $u^i+u_1^s=0$ in $\mathbb{R}^N\backslash\overline{\Omega}$, which contradicts with the fact that
\[
\lim_{|x|\rightarrow\infty}|u_1^s(x)+e^{ikx\cdot d}|=1.
\]
Hence, $(u_1, u_2)$ is a pair of nontrivial solutions to \eqref{eq:itep}. According to \cite{CM}, $k$ is {\it an interior transmission eigenvalue} to \eqref{eq:itep} with $u_1$ and $u_2$ being the corresponding eigenfunctions.

Based on the above observation, one can show the uniqueness \eqref{eq:uniqueness} with a fixed $d\in\mathbb{S}^{N-1}$ and all $(\hat x, k)\in\mathbb{S}^{N-1}\times\mathbb{R}_+$ by absurdity as follows. If $u^1_\infty(\hat x; d, k)=u^2_\infty(\hat x; d, k)$ for all $(\hat x, k)\in\mathbb{S}^{N-1}\times\mathbb{R}_+$, then by a similar argument to derive \eqref{eq:itep}, we know every $k\in\mathbb{R}_+$ is an interior transmission eigenvalue with $u^1(x; d, k)$ and $u^2(x; d, k)$ being the corresponding eigenfunctions. If one can show that the interior transmission eigenvalues to \eqref{eq:itep} for certain admissible $n_1$ and $n_2$ are discrete, then one obviously arrives at a contradiction. However, such discreteness of the interior transmission eigenvalues for \eqref{eq:itep} is only available for certain restricted $n_1$ and $n_2$ in the literature. The connection discussed above obviously bridges the study on uniqueness of the inverse scattering problem and that on the interior transmission eigenvalue problems, which in our hope might open up some new directions of research in the relevant field.

Next, we present a local uniqueness result in determining a refractive index by using the fixed-incident-direction scattering amplitude.

\begin{thm}\label{thm:1}
Let $(\Omega, n)$ be an inhomogeneous medium with $n_*<n(x)<n^*$ for $x\in\Omega$ and let $\varepsilon(x)\in L^\infty(\Omega)$. Let $u_\infty$ and $v_\infty$ denote the scattering amplitudes corresponding to $(\Omega, n)$ and $(\Omega, n+\varepsilon)$, respectively. If there exisit two positive constants $\epsilon^+$ and $\epsilon^-$, and a neighborhood of $\partial \Omega$, $neigh(\partial\Omega)$, such that $\epsilon^+<\epsilon^-<n_*$, $\varepsilon(x)\geq -\epsilon^-$ for $x\in\Omega$, and either $\varepsilon(x)\geq \epsilon_+$ or $\varepsilon(x)\leq -\epsilon_+$ for $x\in neigh(\partial\Omega)$.Then one cannot have that
\begin{equation}\label{eq:local u1}
u_\infty(\hat x; d, k)=v_\infty(\hat x; d, k)
\end{equation}
for any fixed $d\in\mathbb{S}^{N-1}$, and all $(\hat x, k)\in\mathbb{S}^{N-1}\times\mathbb{R}_+$.
\end{thm}

\begin{proof}
Assume that \eqref{eq:local u1} holds for a fixed $d\in\mathbb{S}^{N-1}$ and all $(\hat x, k)\in\mathbb{S}^{N-1}\times \mathbb{R}_+$. Then, according to our earlier discussion, we know that every $k\in\mathbb{R}_+$ is an interior transmission eigenvalue to
\begin{equation}\label{eq:local u2}
\begin{cases}
\Delta u+k^2 n^2 u=0\qquad &\mbox{in\ \ $\Omega$},\\
\Delta v+k^2(n+\varepsilon)^2 v=0\qquad &\mbox{in\ \ $\Omega$},\\
u=v,\ \ \frac{\partial u}{\partial\nu}=\frac{\partial v}{\partial\nu} \ \ &\mbox{on\ \ $\partial\Omega$}.
\end{cases}
\end{equation}
However, in \cite{S}, it is shown that provided $n$ and $\varepsilon$ satisfy the assumptions stated in the theorem, the interior transmission eigenvalue problem \eqref{eq:local u2} possesses (at most) a discrete set of eigenvalues, which immediately yields a contradiction.

The proof is completed.
\end{proof}

\section{Uniqueness in determining a constant refractive index}\label{sect:3}

In this section, we shall prove that an inhomogeneous medium with a constant refractive index can be uniquely determined by a single far-field measurement. Indeed, we have

\begin{thm}\label{thm:onem}
Let $(\Omega, n)$ be an inhomogeneous medium with the refractive index $n$ being a complex-valued constant such that $\Im n\geq 0$ and $0< |n|\leq n^*$, and let $u_\infty(\hat x; d, k)$ be the associated scattering amplitude. Then there exists a positive constant $k_0$, depending only on $n^*$ and $\Omega$, such that $n$ is uniquely determined by $u_\infty(\hat x; d, k)$ with any fixed $0<k<k_0$, $d\in\mathbb{S}^{N-1}$, and all $\hat x\in\mathbb{S}^{N-1}$.
\end{thm}

\begin{rem}
In Theorem~\ref{thm:onem}, we allow $n$ to be complex-valued. $\Im n$ is known as absorbing or damping coefficient of the acoustic medium $(\Omega, n)$.
\end{rem}

\begin{proof}
As before, we shall prove the theorem by absurdity. Let $(\Omega, \widetilde n)$ be another inhomogeneous medium with the refractive index $\widetilde n$ being a constant such that $\Im \widetilde n\geq 0$, $0<|\widetilde n|\leq n^*$, $\widetilde n\neq n$ and
\begin{equation}\label{eq:bb}
u_\infty(\hat x; d, k)=v_\infty(\hat x; d, k)\quad \mbox{for fixed $k>0$ and $d\in\mathbb{S}^{N-1}$, and all $\hat x\in\mathbb{S}^{N-1}$ },
\end{equation}
where $v_\infty(\hat x; d, k)$ denote the scattering amplitude corresponding to $(\Omega, \widetilde n)$. By a similar argument to that in deriving \eqref{eq:itep}, we know $(u, v)\in H^1(\Omega)\times H^1(\Omega)$ is a pair of interior transmission eigenfunctions to the following system
\begin{equation}
\begin{cases}
\Delta u+k^2 n^2 u=0\qquad &\mbox{in\ \ $\Omega$},\\
\Delta v+k^2 \widetilde n^2 v=0\qquad &\mbox{in\ \ $\Omega$},\\
u=v,\ \ \frac{\partial u}{\partial \nu}=\frac{\partial v}{\partial \nu}\ &\mbox{on\ \ $\partial\Omega$}.
\end{cases}
\end{equation}
Set
\[
w:=\frac{u-v}{k^2(\widetilde n^2-n^2)}.
\]
It is straightforward to show that $(w,u)\in H^1(\Omega)\times H^1(\Omega)$ satisfy
\begin{equation}\label{eq:bb2}
\begin{cases}
\Delta w+k^2 n^2 w=v\qquad &\mbox{in\ \ $\Omega$},\\
\Delta v+k^2\widetilde n^2 v=0\qquad &\mbox{in\ \ $\Omega$},\\
w=0,\ \ \frac{\partial w}{\partial\nu}=0\ &\mbox{on\ \ $\partial \Omega$}.
\end{cases}
\end{equation}

Next, multiplying both sides of the first equation in \eqref{eq:bb2} by $\overline{v}$ and then integrating over $\Omega$, we have
\begin{align}
\int_{\Omega}|v|^2\ dx=& \int_{\Omega}(\Delta w+k^2 n^2 w)\cdot \overline{v}\ dx\label{eq:t1}\\
=&\int_{\Omega} (\Delta w+k^2 n^2 w)\cdot\overline{v}\ dx-\int_{\Omega}(\Delta\overline{v}+k^2{\overline{\widetilde n}}^2\overline{v})\cdot w\ dx\label{eq:t2}\\
=&\int_{\partial \Omega}\left(\frac{\partial w}{\partial \nu}\cdot \overline{v}-\frac{\partial \overline{v}}{\partial \nu}\cdot w \right)ds(x)+k^2(n^2-{\overline{\widetilde n}}^2)\int_{\Omega}\overline{v}\cdot w\ dx\label{eq:t3}\\
=& k^2(n^2-{\overline{\widetilde n}}^2)\int_{\Omega}\overline{v}\cdot w\ dx.\label{eq:t4}
\end{align}
From \eqref{eq:t1} to \eqref{eq:t2}, we have made use of the second equation in \eqref{eq:bb2}; from \eqref{eq:t2} to \eqref{eq:t3}, we have made use of Green's formula; and from \eqref{eq:t3} to \eqref{eq:t4}, we have made use of the homogeneous boundary conditions for $w$ in \eqref{eq:bb2}. Now, by \eqref{eq:t1}--\eqref{eq:t4}, we clearly have
\begin{equation}\label{eq:bb3}
\int_{\Omega}|v|^2\ dx\leq 2k^2{n^*}^2\|v\|_{L^2(\Omega)}\|w\|_{L^2(\Omega)}.
\end{equation}
By Lemma~\ref{lem:auxi} in the following, we know that provided $k_0$ is sufficiently small,
\begin{equation}\label{eq:bb4}
\|w\|_{L^2(\Omega)}\leq C(\Omega)\|v\|_{L^2(\Omega)},
\end{equation}
where $C(\Omega)$ is a positive constant depending only on $\Omega$. Finally, we further require that
\begin{equation}\label{eq:bb5}
k_0<\frac{1}{\sqrt{2C(\Omega)}\;\;n^*}.
\end{equation}
Then, by \eqref{eq:bb3}--\eqref{eq:bb5}, it is straightforward to show that $\|v\|_{L^2(\Omega)}=0$, which immediately yields a contradiction.

The proof is completed.
\end{proof}

\begin{lem}\label{lem:auxi}
Let $f\in L^2(\Omega)$ and $w\in H^1(\Omega)$ satisify
\begin{equation}\label{eq:cc1}
\Delta w+k^2 w=f\ \ \ \mbox{in\ \ $\Omega$},\ \ w=0\ \ \mbox{on\ \ $\partial\Omega$}.
\end{equation}
Then there exist positive constants $k_0=k_0(\Omega)$ and $C=C(\Omega)$ such that 
\begin{equation}\label{eq:cc2}
\|w\|_{L^2(\Omega)}\leq C(\Omega) \|f\|_{L^2(\Omega)},\quad\mbox{for all}\quad k<k_0(\Omega).
\end{equation}
\end{lem}
\begin{proof}
By integrating by parts, we first have
\begin{equation}\label{eq:cc3}
\begin{split}
&\int_{\Omega} f\cdot\overline{w}\ dx
=\int_{\Omega} (\Delta w+k^2 w)\cdot \overline{w}\ dx\\
=&-\int_{\Omega}|\nabla w|^2\ dx+k^2\int_{\Omega}|w|^2\ dx.
\end{split}
\end{equation}
Next, by Poincar\'e inequality, we have
\begin{equation}\label{eq:cc4}
\int_{\Omega} |w|^2\ dx\leq C_1(\Omega)\int_{\Omega}|\nabla w|^2\ dx,
\end{equation}
where $C_1(\Omega)$ is a positive constant depending only on $\Omega$. Then, by \eqref{eq:cc3} and \eqref{eq:cc4}, we further have
\begin{equation}\label{eq:cc5}
\begin{split}
\frac{1}{C_1(\Omega)}\int_{\Omega}|w|^2\ dx\leq & \int_{\Omega}|\nabla w|^2\ dx\leq k^2\int_{\Omega}|w|^2\ dx+|\int_{\Omega} f\cdot\overline{w}\ dx|\\
\leq & k^2\int_{\Omega}|w|^2 \ dx+\alpha\int_{\Omega} |w|^2\ dx+\frac{1}{4\alpha} \int_{\Omega} |f|^2\ dx\\
\leq & k_0^2\int_{\Omega}|w|^2 \ dx+\alpha\int_{\Omega} |w|^2\ dx+\frac{1}{4\alpha} \int_{\Omega} |f|^2\ dx,
\end{split}
\end{equation}
where $\alpha\in\mathbb{R}_+$. By choosing
\begin{equation}\label{eq:cc6}
k_0=\frac{1}{2\sqrt{C_1(\Omega)}}
\end{equation}
and letting $\alpha=k_0^2$ in \eqref{eq:cc5}, we can compute directly that
\begin{equation}\label{eq:cc7}
\int_{\Omega}|w|^2\ dx\leq 2C_1(\Omega)^2\int_{\Omega}|f|^2\ dx.
\end{equation}
Therefore, the lemma is proved by taking $C(\Omega)=\sqrt{2}C_1(\Omega)$.

\end{proof}

In the rest of this section, we present a uniqueness result in determining a spherically symmetric refractive index by a single far-field measurement without the smallness condition on $k$ in Theorem~\ref{thm:onem}.

\begin{thm}\label{Th}
Let $(\Omega, n)$ be an inhomogeneous medium with a constant refractive index $n\neq 0$. Suppose further that $\Omega=B_R(z):=\{x\in\R^N: |x-z|=R\}$ is a ball centered at $z\in\R^N$ with radius $R>0$. Then, the ball (that is, $R$ and $z$) and its refractive index $n$ can be uniquely determined by the far-field pattern $u_\infty(\hat{x}; d, k)$ for all $\hat{x}\in \mathbb{S}^{N-1}$ with any fixed incident direction $d\in \mathbb{S}^{N-1}$ and wave number $k\in \R_+$.
\end{thm}
\begin{proof}
We carry out the proof following the argument in \cite[Theorem 5.4]{CK} for the unique determination of sound-soft balls.  Let $(\tilde{\Omega}, \tilde{n})$ be another spherically symmetric medium with the constant refractive index $\tilde{n}$, where $\tilde{\Omega}=B_{\tilde{R}}(\tilde{z})$. Denote by $\tilde{u}_\infty(\hat{x}; d, k)$ the far-field pattern corresponding to $\tilde{\Omega}$ incited by the plane wave $u^{i}=\exp(ikx\cdot d)$. Assuming $\tilde{u}_\infty(\hat{x}; d, k)=u_\infty(\hat{x}; d, k)$ for all $\hat{x}\in \mathbb{S}^{N-1}$ ($N=2,3$), we shall prove $z=\tilde{z}$, $R=\tilde{R}$ and $n=\tilde{n}$. By Rellich' s Lemma, the scattered waves $u^{s}(x; B_R(z))$ and $\tilde{u}^{s}(x,B_{\tilde{R}}(\tilde{z}))$ coincide in $\R^N\backslash\overline{B_R(z)\cup B_{\tilde{R}}(\tilde{z}) }$.

We first
prove  $z=\tilde{z}$, i.e., the centers of $\Omega$ and $\tilde{\Omega}$ coincide.
Similar to the proof of \cite[Theorem 5.4]{CK}, we claim that the scattered field $u^{s}(x; B_R(z))$ can be analytically extended from $\R^3\backslash\overline{\Omega}$ into $\R^3\backslash\{z\}$. In fact, this can be derived from the explicit expression of $u^{s}$ in terms of $k, n, R$ and $z$. Without loss of generality we suppose $N=3$. We make an ansatz on the scattered field $u^{s}(x; B_R(z))$ and the transmitted wave $u(x; B_R(z))$ that
\be\label{ansatz}\begin{split}
u^{s}(x; B_R(z))&=&\sum_{m=0}^\infty i^m (2m+1)\,A_m\,h_m^{(1)}(k|x-z|) P_m(\cos\theta),\quad A_m\in \C,\quad |x|>R,\\
u(x; B_R(z))&=& \sum_{m=0}^\infty i^m (2m+1)\,B_m\,j_m(kn|x-z|) P_m(\cos\theta),\quad B_m\in \C,\quad|x|<R,
\end{split}
\en
where $h_m^{(1)}$ denotes the Hankel function of the first kind of order $m$, $j_m$ the Bessel functions of order $m$, $P_m$ the Legendre polynomials and $\theta$ the angle between $d$ and $x-z$. Recalling the Jacobi-Anger expansion (see e.g., \cite[(2.46)]{CK})
\ben
e^{ik x\cdot d}=e^{ikz\cdot d}e^{ik (x-z)\cdot d}=e^{ikz\cdot d}\sum_{m=0}^\infty i^m (2m+1)\,j_m(k|x-z|)\,P_m(\cos\theta),\quad x\in \R^3,
\enn
and taking into account the transmission conditions between $u$ and $u^{s}$ on $|x|=R$, we obtain the following algebraic
equations for $A_m$ and $B_m$:
\ben
\begin{pmatrix}
h_m^{(1)}(t) & -j_m(tn)\\
t\,h_m^{(1)}\!'(t) & -tn\,j_m'(tn)
\end{pmatrix}\begin{pmatrix}
A_m \\ B_m
\end{pmatrix}=-e^{ikz\cdot d}
\begin{pmatrix}
j_m(t) \\ t j_m'(t)
\end{pmatrix},\quad t:=kR.
\enn
Simple calculations show that
\ben
A_n=-e^{ikz\cdot d}\frac{j_m'(t)\, j_m(tn)-n j_m(t)\,j_m'(tn) }{h_m^{(1)}\,'(t)\, j_m(tn)-n\, h_m^{(1)}(t)\,j_m'(tn)}.
\enn
By the asymptotic behavior of the spherical Bessel and Hankel functions as $n\rightarrow \infty$ and their differential formulas (see, e.g., \cite[Chapter 2.4]{CK}), we have
 \ben
  j_m(t)&=&\frac{t^m}{(2m+1)!!}\left(1+\mathcal{O}(\frac{1}{m})\right),\quad j_m'(t)=\frac{nt^{m-1}}{(2m+1)!!}\left(1+\mathcal{O}(\frac{1}{m})\right),\\
     h^{(1)}_m(t)&=&\frac{(2m-1)!!}{i\, t^{m+1}}\left(1+\mathcal{O}(\frac{1}{m})\right),\quad
   h^{(1)}_m\,\!'(t)=-\frac{(m+1)\,(2m-1)!!}{i\,  t^{m+2}}\left(1+\mathcal{O}(\frac{1}{m})\right),
 \enn where $m!!=1\cdot3\cdot5\cdots m$ for any odd number $m\in \N^+$.
Inserting the previous asympotics into the expression of $A_n$ leads to the estimate
\ben
A_m=\mathcal{O}\left(\frac{t^{2m+1}}{(2m+1)!!\;(2m+1)!!}\right)\quad\mbox{as}\quad m\rightarrow\infty,
\enn
and thus
\ben
(2m+1)A_m h_m^{(1)}(k|x-z|)=\mathcal{O}\left(\frac{R^{2m+1}k^m}{|x-z|^{m+1}(2m+1)!!}\right).
\enn
This implies that the scattered field $u^{s}(x,B_R(z))$ converges uniformly in any compact subset of $\R^3\backslash\{z\}$. Analogously,
$\tilde{u}^{s}(x,B_{\tilde{R}}(\tilde{z}))$ has an extension from $\R^3\backslash\overline{B_{\tilde{R}}(\tilde{z})}$ into
$\R^3\backslash\{\tilde{z}\}$. Now, we assume $z\neq\tilde{z}$. Defining $u^{s}(x,B_R(z))|_{x=z}:= \tilde{u}^{s}(z,B_{\tilde{R}}(\tilde{z}))$, we obtain an entire function $u^{s}(x,B_R(z))$ that satisfies the Helmholtz equation and the Sommerfeld radiation condition, leading to $u^{s}(x,B_R(z))=0$ in $\R^3$. This contradiction indicate that one must have $z=\tilde{z}$.

From the series (\ref{ansatz}), one readily concludes that $u_{\infty}$ and  $\tilde{u}_{\infty}$ depend only on the angle $\theta$ between the incident direction $d$ and the observation direction $(x-z)/|x-z|$. Hence, the relation  $\tilde{u}_\infty(\hat{x}; d, k)=u_\infty(\hat{x}; d, k)$ for one incident direction implies the coincidence of far-field patterns for all incident directions. Consequently, we have $R=\tilde{R}$ and $n=\tilde{n}$ due to the uniqueness in the inverse medium problem using all incident and observation directions; see e.g.,\cite[Chapter 10.2]{CK} where the refractive index is allowed to be piecewise continuous. This completes the proof in 3D.

For the two-dimensional case, the determination of the center of the disc $B_R(z)$ can be shown in a completely manner to the 3D case. As soon as the center of $B_R(z)$ is recovered, the uniqueness in determining $R$ and $n$ directly follows from the uniqueness in determining a potential for the 2D Schr\"odinger equation in \cite{B}; see also \cite{BL} and \cite{OYa}. 

The proof is completed.

\end{proof}


By the invariance of the Helmholtz equation under rotations, it is easily verified that for an inhomogeneous medium $(B_R(z), n(|x-z|))$, knowing the scattering amplitude for a fixed incident direction is equivalent to knowing the scattering amplitude for all incident directions. Hence, by a completely similar argument to the proof of Theorem~\ref{Th}, we have

\begin{thm}\label{thm:s1}
Let $\Omega:=B_R(z)$ be a ball of radius $R>0$ centered at $z$ in $\mathbb{R}^N$. Let $(\Omega, n)$ be an inhomogeneous medium with the refractive index $n:=n(|x-z|)\in L^\infty(\Omega)$ being a complex-valued function of $|x-z|$. It is supposed that the center $z$ is known in advance. Then the refractive index $n$ can be uniquely determined by the associated scattering amplitude $u_\infty(\hat x; d, k)$ with any fixed $k\in\mathbb{R}_+$, $d\in\mathbb{S}^{N-1}$ and all $\hat{x}\in\mathbb{S}^{N-1}$.
\end{thm}

\section{Concluding remarks}\label{sect:remark}

One of the classical inverse scattering problems is the so-called {\it obstacle problem}. Let $D$ be a bounded Lipschitz domain such that $\mathbb{R}^N\backslash\overline{D}$ is connected. Consider the following scattering problem
\begin{equation}\label{eq:dd1}
\begin{cases}
\Delta u+k^2 u=0\quad &\mbox{in\ \ $\mathbb{R}^N\backslash\overline{D}$},\\
u=u^i+u^s\quad &\mbox{in\ \ $\mathbb{R}^N\backslash\overline{D}$},\\
u=0\quad &\mbox{on\ \ $\partial D$},\\
\lim_{|x|\rightarrow \infty} |x|^{\frac{N-1}{2}}\left\{ \frac{\partial u^s(x)}{\partial |x|}-ik u^s(x) \right\}=0.
\end{cases}
\end{equation}
The scattered field in \eqref{eq:dd1} possesses the same asymptotic expansion as that in \eqref{eq:farfield}. The inverse obstacle scattering problem is to recover $D$ by knowledge of $u_\infty(\hat x; d, k)$. $D$ is known as a sound-soft obstacle in the physical literature. It has been long conjectured that $D$ can be uniquely determined by a single far-field measurement, namely $u_\infty(\hat x; d, k)$ for all $\hat x\in\mathbb{S}^{N-1}$, but fixed $k\in\mathbb{R}_+$ and $d\in\mathbb{S}^{N-1}$. We note that the inverse obstacle problem is formally posed by a single far-field measurement. The first uniqueness result is due to Schiffer (cf. \cite{Lax}), where the uniqueness was established by infinitely many far-field measurements, namely $u_\infty(\hat x; d, k)$ with all $\hat x\in\mathbb{S}^{N-1}$, and either i).~a fixed $d\in\mathbb{S}^{N-1}$ and infinitely many $k$'s; or ii).~a fixed $k\in\mathbb{R}_+$ and infinitely many $d$'s. The corresponding proof is based on an absurdity argument, and the essential ingredient is the discreteness of the Dirichlet eigenvalues for the negative Laplacian in a bounded domain. In this sense, our argument in Section~\ref{sect:2} on uniquely determining a generic inhomogeneous medium can be taken as a counterpart to Schiffer's argument on uniquely determining a generic sound-soft obstacle. However, in order to establish the uniqueness in determining a generic inhomogeneous medium, one has to deal with the more challenging nonlinear and non-selfadjoint interior transmission eigenvalue problems \eqref{eq:itep}. We conjecture that a generic refractive index $n(x)$ can be uniquely determined the formally-determined fixed-incident-direction scattering amplitude. A possible way to achieve this goal is to study the discreteness of the interior transmission eigenvalues of \eqref{eq:itep}, but one needs to peel off the `singular regions' where $n_1=n_2$ from exterior. We shall present such a study in the future.

By using the fact that there exists a lower bound on the Dirichlet eigenvalues, Colton and Sleeman established the uniqueness with a single far-field measurement for the inverse obstacle problem provided the obstacle is sufficiently small (see \cite{CS}). Our uniqueness in Section~\ref{sect:3} on uniquely determining a constant refractive index by a single far-field measurement can be taken as a counterpart to that uniqueness due to Colton and Sleeman. In order to prove Theorem~\ref{thm:onem}, we actually have shown that there is a lower bound for the positive interior transmission eigenvalues of \eqref{eq:bb}. The lower bound of positive interior transmission eigenvalues was also considered in \cite{CG}. However, our approach works in a more general setting than that was considered in \cite{CG}. Finally,  we note that there are some significant progresses in uniquely determining a generic obstacle by using a single far-field measurement even without the smallness condition; see \cite{AL,CY,LiuZou06,LZ} and the references therein. We believe that a generic constant refractive index can also be uniquely determined by a single far-field measurement without the smallness condition posed in Theorem~\ref{thm:onem}. Indeed, Theorems~\ref{Th} and \ref{thm:s1} cast some light on this conjecture.

\section*{Acknowledgement}
The first author gratefully acknowledges  the support by the  German Research Foundation (DFG) under Grant No. HU 2111/1-1.
The work of Hongyu Liu was supported by NSF grant, DMS 1207784.

\end{document}